\documentclass[12pt]{amsart}
\usepackage{amsfonts,amssymb,amscd,amsmath,enumerate,verbatim,calc}
\usepackage{graphicx}

\newtheorem{Theorem}{Theorem}[section]
\newtheorem{Lemma}[Theorem]{Lemma}
\newtheorem{Corollary}[Theorem]{Corollary}
\newtheorem{Proposition}[Theorem]{Proposition}
\newtheorem{Observation} [Theorem] {Observation}
\newtheorem{Remark}[Theorem]{Remark}

\begin{document}
\title{On The Fixed Number of Graphs}
\author{I. Javaid*, M. Murtaza, M. Asif, F. Iftikhar}
\keywords{Fixing set; Stabilizer; Fixing number; Fixed number
\indent $^*$ Corresponding author: imran.javaid@bzu.edu.pk}
\address{Centre for advanced studies in Pure and Applied Mathematics,
Bahauddin Zakariya University Multan, Pakistan\newline Email:
imran.javaid@bzu.edu.pk, mahru830@gmail.com, asifmaths@yahoo.com,
farah$\_$sethi@yahoo.com}
\date{}
\maketitle
\begin{abstract}
An automorphism on a graph $G$ is a bijective mapping on the vertex
set $V(G)$, which preserves the relation of adjacency between any
two vertices of $G$. An automorphism $g$ fixes a vertex $v$ if $g$
maps $v$ onto itself. The stabilizer of a set $S$ of vertices is the
set of all automorphisms that fix vertices of $S$. A set $F$ is
called fixing set of $G$, if its stabilizer is trivial. The fixing
number of a graph is the cardinality of a smallest fixing set. The
fixed number of a graph $G$ is the minimum $k$, such that every
$k$-set of vertices of $G$ is a fixing set of $G$. A graph $G$ is
called a $k$-fixed graph if its fixing number and fixed number are
both $k$. In this paper, we study the fixed number of a graph and
give construction of a graph of higher fixed number from graph with
lower fixed number. We find bound on $k$ in terms of diameter $d$ of
a distance-transitive $k$-fixed graph.
\end{abstract}
\section{Introduction}
Let $G=(V(G),E(G))$ be a connected graph with order $n$. The degree
of a vertex $v$ in $G$, denoted by deg$_{G}(v)$, is the number of
edges that are incident to $v$ in $G$. We denote by $\Delta(G)$, the
maximum degree and $\delta(G)$, the minimum degree of vertices of
$G$. The \emph{distance} between two vertices $x$ and $y$, denoted
by $d(x,y)$, is the length of a shortest path between $x$ and $y$ in
$G$. The $eccentricity$ of a vertex $x\in V(G)$ is $e(x)=$max$_{y\in
V(G)}d(x,y)$ and the $diameter$ of $G$ is max$_{x\in V(G)}e(x)$. For
a vertex $v\in V(G)$, the \emph{neighborhood} of $v$, denoted by
$N_{G}(v)$, is the set of all vertices adjacent to $v$ in $G$.
\par An automorphism of $G$, $g:V(G)\rightarrow V(G)$, is a
permutation on $V(G)$ such that $g(u)g(v) \in E(G) \Leftrightarrow
uv \in E(G)$, i.e., the relation of adjacency is preserved under
automorphism $g$. The set of all such permutations for a graph $G$
forms a group under the operation of composition of permutations. It
is called the automorphism group of $G$, denoted by $Aut(G)$ and it
is a subgroup of symmetric group $S_n$, the group of all
permutations on $n$ vertices. A graph $G$ with trivial automorphism
group is called rigid or asymmetric graph and such a graph has no
symmetries. In this paper, all the graphs (unless stated otherwise)
have non-trivial automorphism group i.e., $Aut(G)\neq \{id\}$. Let
$u,v\in V(G)$, we say $u$ is $similar$ to $v$, denoted by $u\sim v$
(or more specifically $u\sim^g v$) if there is an automorphism $g\in
Aut(G)$ such that $g(u)=v$. It can be seen that similarity is an
equivalence relation on vertices of $G$ and hence it partitions the
vertex set $V(G)$ into disjoint equivalence classes, called orbits
of $G$. The $orbit$ of a vertex $v$ is defined as
$\mathcal{O}(v)=\{u\in V(G)|u\sim v\}$. The idea of fixing sets was
introduced by Erwin and Harary in \cite{EH}. They used following
terminology: The $stabilizer$ of a vertex $v\in V(G)$ is defined as,
$stab(v)=\{f\in Aut(G)|f(v)=v\}$. The $stabilizer$ of a set of
vertices $F\subseteq V(G)$ is defined as, $stab(F)=\{f\in Aut(G) |
f(v)=v$ \emph{for all} $v\in F\}=\cap_{v\in F}stab(v)$. A vertex $v$
is $fixed$ by an automorphism $g\in Aut(G)$ if $g\in stab(v)$. A set
of vertices $F$ is a $fixing$ $set$ if $stab(F)$ is trivial, i.e.,
the only automorphism that fixes all vertices of $F$ is the trivial
automorphism. The cardinality of a smallest fixing set is called the
\emph{fixing number} of $G$. We will refer a set of vertices
$A\subset V(G)$ for which $stab(A)\setminus \{id\}\neq \emptyset$ as
\emph{non-fixing} set. A vertex $v\in V(G)$ is called a \emph{fixed}
vertex if $stab(v)=Aut(G)$. Every graph has a fixing set. Trivially
the set of vertices itself is a fixing set. It is also clear that a
set containing all but one vertex is a fixing set. Following theorem
gives a relation between orbits and stabilizers.
\begin{Theorem}\label{OrbStabThrm}(Orbit-Stabilizer Theorem)
 Let $G$ be a connected graph and $v\in V(G)$.
  $$|Aut(G)|=|\mathcal{O}(v)| · |stab_{Aut(G)}(v)|.$$
\end{Theorem}

Boutin introduced determining set in \cite{B}. A set $D\subseteq
V(G)$ is said to be a \emph{determining set} for a graph $G$ if
whenever $g,h \in Aut(G)$ so that $g(x)=h(x)$ for all $x\in D$, then
$g(v)=h(v)$ for all $v \in V(G)$. The minimum cardinality of a
determining set of a graph $G$, denoted by $Det(G)$, is called the
\emph{determining number} of $G$. Following lemma given in \cite{CJ}
shows equivalence between definitions of fixing set and determining
sets.
\begin{Lemma}\label{lma1}
\cite{CJ} A set of vertices is a fixing set if and only if it is a
determining set. \end{Lemma} Thus notions of fixing number and
determining number of a graph $G$ are same.
\par
The notion of fixing set is closely related to another well-studied
notion, resolving set, defined in the following way: A vertex $v\in
G$ resolves vertices $x,y\in V(G)$ if $d(x,v)\neq d(y,v)$. A set
$W\subseteq V(G)$ is called a resolving set for $G$ if for every
pair $x,y\in V(G)$, there exists a vertex $w\in W$ such that
$d(x,w)\neq d(y,w)$. The idea of resolving set was introduced by
Slater \cite{SL} and he referred this set as a locating set. The
cardinality of a minimum resolving set in a graph $G$, denoted by
$\beta(G)$, is called the \emph{metric dimension} of $G$. The
\emph{resolving number} $res(G)$ of $G$ is the minimum $k$ such that
every $k$-set of vertices is a resolving set of $G$. The following
proposition was independently proved by Erwin and Harary \cite{EH}
(using fixing sets instead of determining sets) and Boutin \cite{B}.
\begin{Proposition}
\cite{B,EH,FH} If $S\subseteq V(G)$ is a resolving set of $G$ then
$S$ is a fixing set of $G$. In particular, $fix(G)\leq \beta(G)$.
\end{Proposition}
Jannesari and Omoomi have discussed the properties of resolving
graphs and randomly $k$-dimensional graphs in \cite{MJB} and
\cite{ORD} which were based on resolving number and metric dimension
of $G$. In this paper, we define the fixed number of a graph, fixing
graph and $k$-fixed graphs. We discuss properties of these graphs in
the context of fixing sets and the fixing number.

\par
The \emph{fixed number} of a graph $G$, $fxd(G)$, is the minimum $k$
such that every $k$-$set$ of vertices is a fixing set of $G$. It may
be noted that $0 \leq fix(G) \leq fxd(G) \leq n-1 $. A graph is said
to be a \emph{$k$-fixed graph} if $fix(G)=fxd(G)=k$. In this paper,
the fixed number $k$, remains in the focus of our attention. A path
of even order is a 1-fixed graph. Similarly a cyclic graph of odd
order is a 2-fixed graph. We give a construction of a graph with
$fxd(G)=r+1$ from a graph with $fxd(G)=r$ in Theorem
\ref{ThmDtrConstruction}. Also a characterization of $k$-fixed
graphs is given in Theorem \ref{thm4}.

\section{The Fixed Number}
Consider the graph $G$ in Figure 1. It is clear that
$Aut(G)=\{e,(12)(34)(56)\}$.
\begin{figure}[h]
       \centerline
       {\includegraphics[width=6cm]{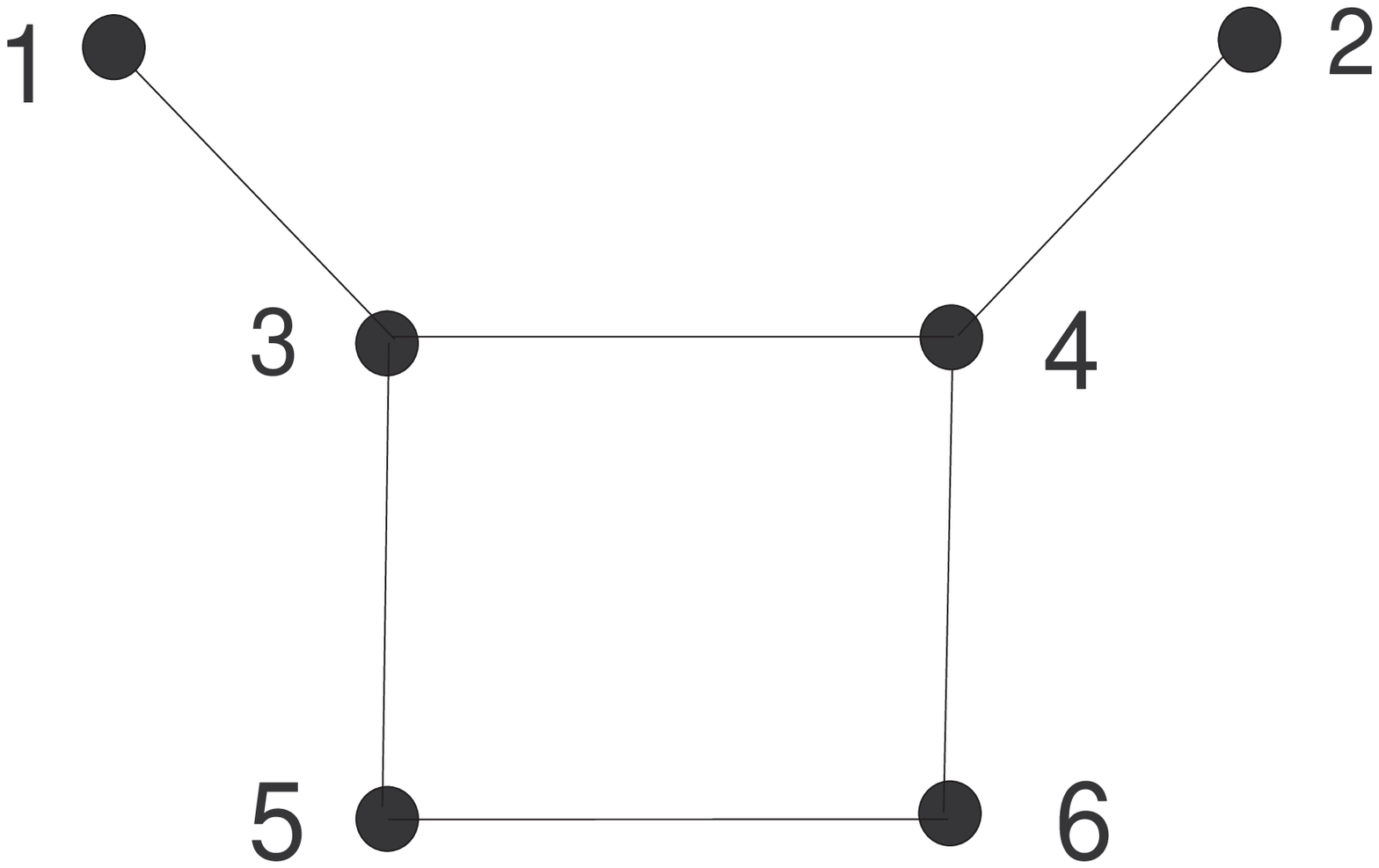}}
        \caption{Graph $G$}\label{fig1}
\end{figure}
Also $stab(v)=\{id\}$ for all $v\in V(G)$. Thus $\{v\}$ for each
$v\in V(G)$ forms a fixing set for $G$. Hence $fix(G)=fxd(G)=1$ and
$G$ is 1-fixed graph. Thus we have following proposition immediately
from definition of fixing set.
\begin{Proposition}\label{PropDtrG1}
Let $G$ be a connected graph and $fxd(G)=1$, then \\
(i) $|\mathcal{O}(v)|=|Aut(G)|$ for all $v\in V(G)$.\\
(ii) $G$ does not have fixed vertices.
\begin{proof}
(i) Since $|stab(v)|=1$ for all $v\in V(G)$ and result follows by
Theorem \ref{OrbStabThrm}. (ii) As $stab(v)=Aut(G)$ for a fixed
vertex $v\in V(G)$ and hence $\{v\}$ does not form a fixing set for
$G$.
\end{proof}
\end{Proposition} The problem of `finding the minimum $k$
such that every $k$-subset of vertices of $G$ is a fixing set of
$G$' is equivalent to the problem of `finding the maximum $r$ such
that there exist an $r$-subset of vertices of $G$ which is not
fixing set of $G$'. Thus, the cardinality of a largest non-fixing
set in $G$ helps in finding the fixed number of $G$. We can see
$r=0$ for the graph $G$ in Figure 1. We have following remarks about
non-fixing sets.
\begin{Remark}\label{NonFixingSets}
Let $G$ be graph of order $n$.\\
(i) If $r$ $(0 \le r \le n-2)$ be the cardinality of a largest
non-fixing subset of $G$, then
$fxd(G)=r+1$.\\
(ii) Let $A$ be a non-fixing set of $G$. For each non-trivial $g\in
stab(A)$ there exist at least one set $B\subset V(G)$ such that
$u\sim^g v$ for all $u,v\in B$.
\end{Remark}

\begin{Proposition}\label{prop3}
Let $G$ be a graph and $u,v \in V(G)$ such that $N(v)\backslash
\{u\}=N(u)\backslash\{v\}$. Let $F$ be a fixing set of $G$, then $u$
or $v$ is in $F$.
\end{Proposition}
\begin{proof}
Let $u,v\in V(G)$ such that $N(v)\backslash
\{u\}=N(u)\backslash\{v\}$. Suppose on contrary both $u$ and $v$ are
not in $F$. As $u$ and $v$ have common neighbors and $u,v\not\in F$
so there exists an automorphism $g\in Aut(G)$ such that $g\in
stab(F)$ and $g(u)=v$. Hence $stab(F)$ has a non-trivial
automorphism, a contradiction.
\end{proof}


\begin{Theorem}\label{thm2}
Let $G$ be a connected graph of order $n$. Then, \\
$fxd(G)=n-1$ if and only if $N(v)\backslash
\{u\}=N(u)\backslash\{v\}$ for some $u,v\in V(G)$.
\end{Theorem}
\begin{proof}
Let $u,v\in V(G)$ such that $N(v)\backslash
\{u\}=N(u)\backslash\{v\}$. Suppose on the contrary that $fxd(G)\leq
n-2$, then $V(G)\backslash \{u,v\}$ is a fixing set for $G$. But, by
Proposition \ref{prop3}, every fixing set contains either $u$ or
$v$.
This contradiction implies that, $fxd(G)=n-1$.\\
Conversely, let $fxd(G)=n-1$. Thus, there exists a non-fixing subset
$T$ of $V(G)$ with $|T|=n-2$. Assume $T=V(G)\setminus\{u,v\}$ for
some $u,v \in V(G)$. Our claim is that $u,v$ are those vertices of
$G$ for which $N(u)\setminus \{v\}=N(v)\setminus \{u\}$. Suppose on
contrary $N(u)\setminus \{v\}\neq N(v)\setminus \{u\}$, then there
exists a vertex $w\in T$ such that $w$ is adjacent to one of the
vertices $u$ or $v$. Without loss of generality, let $w$ is adjacent
to $u$ but not adjacent to $v$. Let a non-trivial automorphism $g\in
stab(T)$ (such a non-trivial automorphism exists because $T$ is not
a fixing set). Since $g$ is non-trivial and $V(G)\setminus
T=\{u,v\}$, so $g(u)=v$. But $u$ cannot map to $v$ under $g$,
because $g\in stab(w)$ and $w$ is adjacent $u$ and not adjacent to
$v$. Hence $g$ also fixes $u$ and $v$, i.e., $g\in stab\{u,v\}$ and
consequently $g$ becomes trivial. Hence $stab(T)$ is trivial, a
contradiction. Thus $N(u)\setminus \{v\}=N(v)\setminus \{u\}$.
\end{proof}
The following theorem given in \cite{ODN} is useful for the proof of
Corollary \ref{cor1}.

\begin{Theorem}\label{thm1}
\cite{ODN} Let $G$ be a connected graph of order $n$. Then \\
 $fix(G)=n-1$ if and only if $G=K_{n}.$
\end{Theorem}

\begin{Corollary}\label{cor1}
Let $G$ be a graph of order $n$ and $G\neq K_n$. If $G$ is
$(n-1)$-fixed graph, then for each pair of distinct vertices $u,v
\in V(G)$, $N(u)\backslash \{v\}\neq N(v)\backslash \{u\}$.
\end {Corollary}
\begin{proof}
Let $N(u)\backslash \{v\}= N(v)\backslash \{u\}$ for some $u,v \in
V(G)$. Then by Theorem \ref{thm2}, $fxd(G)=n-1$. Since $G\neq K_n$,
so by Theorem \ref{thm1} $fix(G)\neq n-1=fxd(G)$. Hence $G$ is not
$(n-1)$-fixed.
\end{proof}

The fixing polynomial, $F(G,x)=\sum_{i=fix(G)}^{n}{\alpha_i x^i}$,
of a graph $G$ of order $n$ is a generating function of sequence
$\{\alpha_i\}$ $(fix(G) \leq i \leq n)$, where $\alpha_i$ is the
number of fixing subset of $G$ of cardinality $i$. For more detail
about fixing polynomial see \cite{IFS} where we discussed properties
of fixing polynomial and found it for different families of graphs.
For example $F(C_3,x)=x^3+3x^2$.

\begin{Theorem}
Let $G$ be a $k$-fixed graph of order $n$. $$
F(G,x)=\sum_{i=k}^{n}{{n\choose i} x^i} $$ \end{Theorem}

\begin{proof}
 Since $fix(G)=fxd(G)=k$ and superset of a fixing set is also a
 fixing set, therefore each $i$-subset $(k \leq i \leq n)$ is a
 fixing set. Hence $\alpha_i= {n \choose i}$ for each $i$, $(k \leq i \leq
 n)$. \end{proof}

\begin{Theorem}\label{ThmDtrConstruction}
Let $G$ be a graph of order $n$ and $fxd(G)=r$. We can construct a
graph $G'$ of order $n+1$, from $G$ such that $fxd(G')=r+1$.
\end{Theorem}
\begin{proof}
Since $fxd(G)=r$, so $G$ has a largest non-fixing set $A$ of
cardinality $|A|=r-1$. By Remark \ref{NonFixingSets}(ii) for each
non-trivial $g\in stab(A)$, there exist at least one set $B\subset
V(G)$ such that $u\sim^g v$ for all $u,v\in B$. Consider
$B=\{v_1,v_2,...,v_l\}$. Take a $K_1=\{x\}$ and join $x$ with
$v_1,v_2,...,v_l$ by edges $xv_1,xv_2,...,xv_l$. We call new graph
$G'$. This completes construction of $G'$. We will now find a
largest non-fixing subset of $G'$. Since $v_i\sim^g v_j$ $(i\neq j,
1\leq i,j \leq l)$ in $G$ and $x$ is adjacent to $v_1,v_2,...,v_l$
in $G'$. So we can find a $g'\in Aut(G')$ such that
$$ g'(u) = \left\{
            \begin{array}{ll}
           \, x &         \,\,\ \mbox{if}\,\ u=x, \\
           \, g(u) &\,\,\ \mbox{if}\,\ u\neq x \\
            \end{array}
             \right.
$$
in $G'$. Clearly, $g'\in stab(x)\cap stab(A)=stab(\{x\}\cup A )$ and
$v_i\sim^{g'} v_j$ $(i\neq j, 1\leq i,j \leq l)$ in $G'$. Since $g'$
is non-trivial and $A$ is a largest non-fixing set in $G$, so $A\cup
\{x\}$ is a largest non-fixing set in $G'$. Hence by Remark
\ref{NonFixingSets}(i), $fxd(G')=|A\cup\{x\}|+1=r+1$
\end{proof}

The following lemma is useful for finding the fixing number of a
tree.
\begin{Lemma}\label{lma1}
\cite{EH}Let $T$ be a tree and $F\subset V(T)$, then $F$ fixes $T$
if and only if $F$ fixes the end vertices of $T$.
\end{Lemma}
\begin{Theorem}\label{thm5}
For every integers $p$ and $q$ with $2\leq p \leq q$, there exists a
graph $G$ with $fix(G)=p$ and $fxd(G)=q$.
\end{Theorem}
\begin{proof}
For $p=q$, $G=K_{p+1}$ will have the desired property. So we
consider $2\leq p < q$. Consider a graph $G$ obtained from a path
$w_{1},w_{2},...,w_{q-p}$. Add $p+1$ vertices
$u_{1},u_{2},...,u_{p+1}$ and $p+1$ edges
$w_{1}u_{1},w_{1}u_{2},...,w_{1}u_{p+1}$ with $w_{1}$. Thus $|V(G)|
= q+1$. Consider set $F\subset V(G)$, $F=\{u_{1},u_{2},...,u_{p}\}$
, then $F$ fixes the set of end vertices
$\{u_{1},u_{2},...,u_{p},u_{p+1}\}$ of $G$. As $G$ is a tree and
$w_{p-q}$ is a fixed end vertex, hence $F$ fixes $G$ by Lemma
\ref{lma1}. Since $F$ is a minimum fixing set for $G$, so
$fix(G)=|F|=p$. Also $fxd(G)=q$ because
$U=\{w_{1},w_{2},...,w_{q-p},u_{1},u_{2},...,u_{p-1}\}$ is the
largest non-fixing set with cardinality $q-1$.
\end{proof}

\section{The Fixing Graph}

\par Let $G$ be a connected graph. The set of fixed vertices of
 $G$ has no contribution in constructing the fixing sets of
  $G$, therefore we define a vertex set $S(G)=\{v\in V(G): v\sim u$
  for some $u(\neq v)\in V(G)\}$ (set of
all vertices of $G$ which are more than one in their orbits). Also
consider $V_s(G)=\{(u,v):u\sim v$ $(u\neq v)$ and $u,v\in V(G)\}$.
Also, if $G$ is an asymmetric graph, then $V_s(G)=\emptyset$. Let
$x\in V(G)$, an arbitrary automorphism $g\in stab(x)$ is said to
\emph{fix a pair} $(u,v)\in V_s(G)$, if $u\not\sim^g v$. If
$(u,v)\not \in V_s(G)$, then $u\not\sim v$ and hence question of
fixing pair $(u,v)$ by a $g\in stab(x)$ does not arise. In this
section, we use $r$ and $s$ to denote $|S(G)|$ and $|V_s(G)|$
respectively. It is clear that $r\le n$ and $\frac{r}{2}\le s\le {r
\choose 2}\le {n\choose 2}$ where $s$ attains its lower bound in
later inequality in case when $r$ is even and pair $(u,v)$ is only
fixed by automorphisms in $stab\{u,v\}$ for all $(u,v)\in V_s(G)$.
Consider the graph $G_1$ in Figure 2 where $r=6$ and $s=7$. $G_1$
has a fixed vertex $v_1$ and $S(G_1)=\{v_2,v_3,v_4,v_5,v_6,v_7\}$
and
$V_s(G_1)=\{(v_2,v_3),(v_4,v_5),(v_4,v_6),(v_4,v_7),(v_5,v_6),(v_5,v_7),(v_6,v_7)\}$.
Since superset of a fixing set is also a fixing set, so we are
interested in fixing set of minimum cardinality. Following remarks
tell us the relation between a fixing set $F$ and $S(G)$.
\begin{Remark}\label{RemarkSG}
Let $G$ be a graph. A set $F\subset V(G)$ is a minimum
 fixing set of $G$, if $F\subset S(G)$ and an arbitrary $g\in stab(F)$ fixes $S(G)$.
\end{Remark}
 The $Fixing$ $Graph$, $D(G)$, of a
graph $G$ is a bipartite graph with bipartition $(S(G),V_s(G))$. A
vertex $x\in S(G)$ is adjacent to a pair $(u,v)\in V_s(G)$ if
$u\not\sim^g v$ for $g\in stab(x)$.
\begin{figure}[h]
       \centerline
       {\includegraphics[width=12cm]{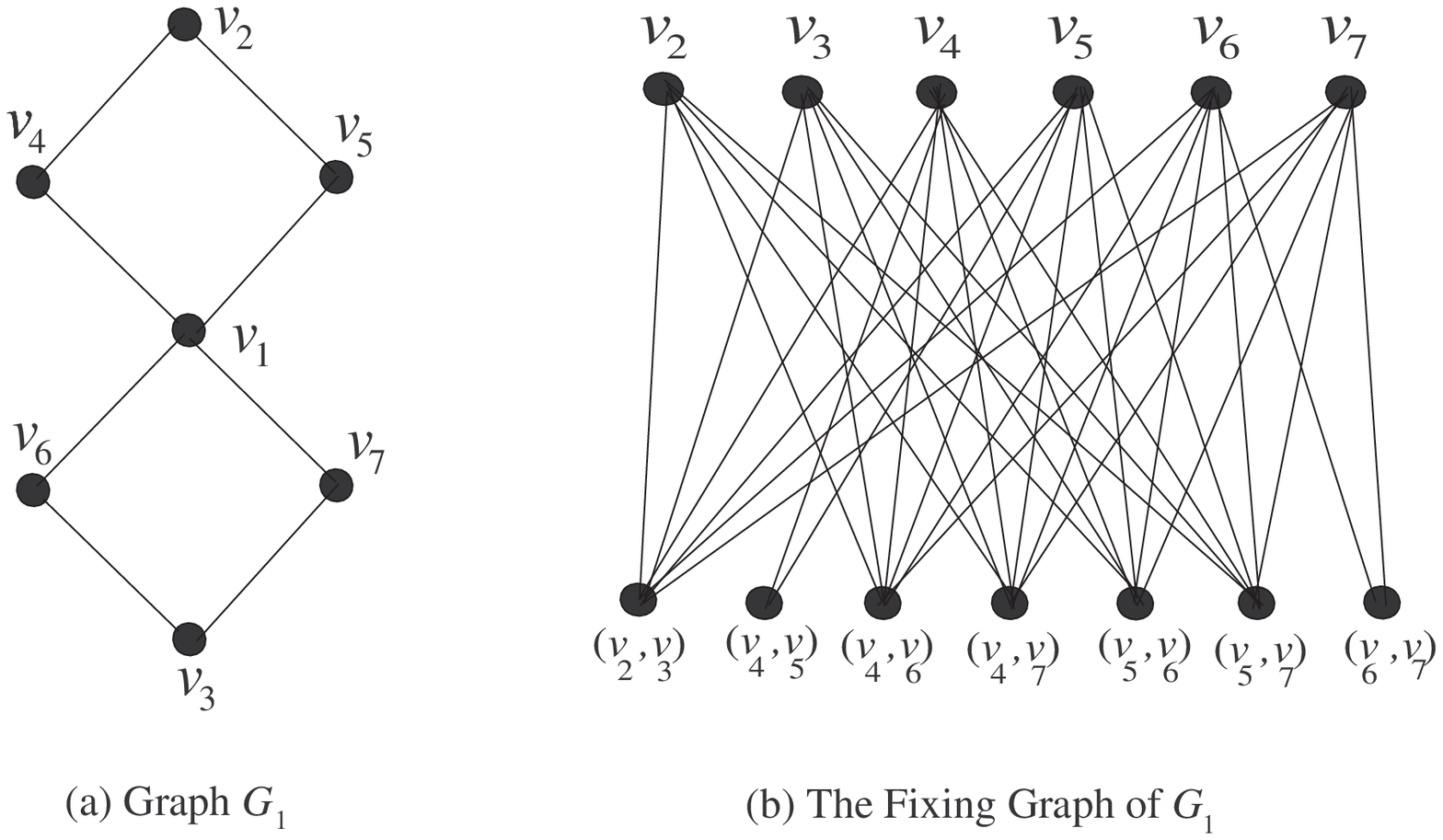}}
        \caption{}\label{fig1}
\end{figure}
Let $F\subseteq S(G)$, then $N_{D(G)}(F)=\{(x,y)\in V_{s}(G) |$
$x\not\sim^g y$ for $g\in stab(F)\}$. In the fixing graph, $D(G)$,
the minimum cardinality of a subset $F$ of $S(G)$ such that
$N_{D(G)}(F)=V_{s}(G)$ is the fixing number of $G$. Figure 2(b)
shows the fixing graph of graph $G_1$ given in Figure 2(a). As
$N_{D(G_1)}\{v_4,v_6\}= V_s(G_1)$, thus $\{v_4,v_6\}$ is a fixing
set of $G_1$ and hence $fix(G_1)=2$.

\begin{Remark}\label{PropFixSetDetermingGraph}
Let $G$ be graph and $F\subset S(G)$ be a fixing set of $G$, then
$N_{D(G)}(F)=V_s(G)$.
\end{Remark}

Also $\{v_1,v_2,v_3,v_4,v_5\}$ is a largest non-fixing set of $G_1$.
In fact every largest non-fixing set must have fixed vertex $v_1$.
So we have following proposition
\begin{Proposition}\label{PropLargestNonFixing}
Let $G$ be a graph and $A$ be a largest non-fixing subset of $G$.
Then $A$ contains all fixed vertices of $G$.
\end{Proposition}
\begin{proof}
Let $x\in V(G)$ be an arbitrary fixed vertex of $G$. Suppose on
contrary $x\not\in A$. Then $stab(A\cup \{x\})=stab(A)\cap
stab(x)=stab(A)\cap Aut(G)=stab(A)\neq \{id\}$ ($A$ is non-fixing
set). Consequently $A\cup \{x\}$ is non-fixing set, a contradiction
that $A$ is largest non-fixing set.
\end{proof}
 Let $t$ be the minimum number such that $1\le t\le r$ and every
$t$-subset $F$ of $S(G)$ has $N_{D(G)}(F)= V_s(G)$, then $t$ is
helpful in finding the fixed number of a graph $G$. The following
theorem gives a way of finding fixed number of a graph using its
fixing graph.
\begin{Theorem}\label{dtrViaDtrmingGraph}
Let $G$ be a graph of order $n$ and $t$ $(1\le t\le r)$ be the
minimum number such that every $t$-subset of $S(G)$ has neighborhood
$V_s(G)$ in $D(G)$. Then
$$ fxd(G)=t+|V(G)\setminus S(G)|$$
\end{Theorem}
\begin{proof}
By Remark \ref{NonFixingSets}(i), we find a largest non-fixing
subset $T$ of $V(G)$. By Proposition \ref{PropLargestNonFixing}
$V(G)\setminus S(G)$ is a subset of largest non-fixing set $T$.
Moreover, by hypothesis there is a $(t-1)$-subset $U$ of $S(G)$ such
that $N_{D(G)}(U)\neq V_s(G)$. Then $U$ is non-fixing set for $G$
and hence $\{V(G)\setminus S(G)\}\cup U$ is a non-fixing set. Also
$\{V(G)\setminus S(G)\}\cup U$ is a largest non-fixing set of $G$,
because by hypothesis, a $t$-subset of $S(G)$ forms a fixing set of
$G$. Further $\{V(G)\setminus S(G)\}\cap U=\emptyset$. Hence by
Remark \ref{NonFixingSets}(i), $$ fxd(G)=|V(G)\setminus
S(G)|+|U|+1=|V(G)\setminus S(G)|+t $$

\end{proof}
In \cite{IJH}, we found an upper bound on the cardinality of edge
set $|E(D(G))|$ of fixing graph $D(G)$ of a graph $G$.
\begin{Proposition}\label{prop1}\cite{IJH}
Let $G$ be a $k$-fixed graph of order $n$, then
\begin{equation}\label{eq5} |E(D(G))|\leq n({n \choose 2}-k+1).\end{equation}
\end{Proposition}
Now we find lower bound on $|E(D(G))|$.
\begin{Proposition}\label{prop1}
If $G$ is a $k$-fixed graph of order $n$, then \[
(\frac{r}{2})(r-k+1)\leq |E(D(G))| \]
\end{Proposition}
\begin{proof}
Let $z\in V_{s}(G)$ and $A$ be a set of vertices of $S(G)$ which are
not adjacent to $z$. Since $N_{D(G)}(A)\neq V_{s}(G)$, therefore $A$
is a non-fixing set of $G$. Our claim is $deg_{D(G)}(z)\ge r-k+1$.
Suppose $deg_{D(G)}(z)\leq r-k$, then $|A|\geq k$, which contradicts
that $fxd(G)=k$ ($A$ is non-fixing set with $|A|\ge k$). Thus,
$deg_{D(G)}(z)\geq r-k+1$ and consequently,
\begin{equation}\label{eq8}
(\frac{r}{2})(r-k+1)\leq s(r-k+1)\leq |E(D(G))|.
\end{equation}
\end{proof}
Thus, on combining (\ref{eq5}) and (\ref{eq8}) we get
\begin{equation}\label{eq7}
(\frac{r}{2}) (r-k+1)\leq |E(D(G))| \leq n({n \choose 2}-k+1).
\end{equation}

\begin{Theorem}\label{thm4}
If $G$ is a $k$-fixed graph and $|S(G)|=r$, then $k\leq 3$ or $k
\geq r-1 $.
\end{Theorem}
\begin{proof}
For each $R\subseteq S(G)$, let
$\overline{N}_{D(G)}(R)=V_{s}(G)\backslash N_{D(G)}(R)$. We claim
that, if $R,T\subseteq S(G)$ with $|R|=|T|=k-1$ and $R\neq T$, then
$\overline{N}_{D(G)}(R)\cap \overline{N}_{D(G)}(T)=\emptyset$.
Otherwise, there exists a pair $\{y,z\} \in
\overline{N}_{D(G)}(R)\cap \overline{N}_{D(G)}(T)$. Therefore,
$\{y,z\} \notin N_{D(G)}(R\cup T)$ and hence, $R\cup T$ is not a
fixing set of $G$. Since, $R\neq T$, $|R\cup T|>|T|=k-1$, which
contradicts that
$fxd(G)=k$. Thus, $\overline{N}_{D(G)}(R)\cap \overline{N}_{D(G)}(T)=\emptyset$.\\
Since, $fix(G)=k$, for each $R\subseteq S(G)$ with $|R|=k-1$,
$\overline{N}_{D(G)}(R)\neq \emptyset$. Now, let
$\Omega=\{R\subseteq S(G):$ $|R| =k-1\}$. Therefore,
\[|\bigcup_{R\in \Omega}\overline{N}_{D(G)}(R)|=\sum_{R\in\Omega}|\overline{N}_{D(G)}(R)|\geq
\sum_{R\in\Omega} 1 = {r \choose k-1}\] On the other hand,
$\bigcup_{R\in \Omega} \overline{N}_{D(G)}(R)\subseteq V_{s}(G)$.
Hence, $|\bigcup_{R\in \Omega}\overline{N}_{D(G)}(R)|\leq s \leq
{r\choose 2}$. Consequently, ${r \choose k-1} \leq {r \choose 2}$.
If $r \leq 4$, then $k \leq 3$. Now, let $r \geq 5$. Thus, $2 \leq
\frac{r+1}{2}$. We know that for each $a,b \leq \frac{n+1}{2}$,
${r\choose a}\leq {r\choose b}$ if and only if $a\leq b$. Therefore,
if $k-1 \leq \frac{r+1}{2}$, then $k-1 \leq 2$, which implies $k
\leq 3$. If $k-1 \geq \frac{r+1}{2}$, then $r-k+1 \leq
\frac{r+1}{2}$. Since ${r \choose r-k+1} = {r \choose k-1}$, we have
${r \choose r-k+1} \leq {r\choose 2}$ and consequently, $r-k+1 \leq
2$, which yields $k\geq r-1$.
\end{proof}
\section{The Distance-Transitive Graph}
We now study the fixed number in a class of graphs known as the
distance-transitive graphs. A graph $G$ is called
distance-transitive if $u,v,x,y\in V(G)$ satisfying $d(u,v)=d(x,y)$,
then there exist an automorphism $g\in Aut(G)$ such that $u\sim^g x$
and $v\sim^g y$. For example, the complete graph $K_n$, the cyclic
graph $C_n$, the Petersen graph, the Johnson graph etc, are
distance-transitive. For more about distance transitive graphs see
\cite{NBiggs}. In this section, we use terminology as described in
section 3 related to the fixing graph $D(G)$ of a graph $G$.
Following proposition given in \cite{NBiggs} tells that the distance
transitive graph does not have fixed vertices.
\begin{Proposition}\cite{NBiggs}
A distance-transitive graph is vertex transitive.
\end{Proposition}
Thus, if $G$ is a distance transitive graph, then $S(G)=V(G)$, $r=n$
and $V_s(G)$ consists of all $n\choose 2$ pairs of vertices of $G$
(i.e., $s= {n\choose 2}$). \begin{Corollary} Let $G$ be a
distance-transitive graph of order $n$. If $G$ is $k$-fixed, then
$k\leq 3$ or $k \geq n-1 $.
\end{Corollary}
\begin{proof}
Since $r=n$ for a distance-transitive graph, so result follows from
Theorem \ref{thm4}.
\end{proof}
Moreover an expression for bounds on $|E(D(G))|$ of a
distance-transitive and $k$-fixed graph $G$ can be obtained by
putting $r=n$ and $s={n\choose 2}$ in (\ref{eq8}) and use the result
in (\ref{eq7}), we get
\begin{equation}\label{eq9}
{n\choose 2} (n-k+1)\leq |E(D(G))| \leq n({n \choose 2}-k+1).
\end{equation}
Also, the following two results given in \cite{MJB} are useful in
our later work.
\begin{Observation}\label{obs3}
\cite{MJB} Let $n_{1},...,n_{r}$ and n be positive integers, with
$\sum_{i=1}^{r} n_{i}=n$. Then, $\sum_{i=1}^{r} {n_{i} \choose 2}$
is minimum if and only if $|n_{i}-n_{j}|\leq 1$, for each $1\leq
i,j\leq r$.
\end{Observation}
\begin{Lemma}\label{lma2}
\cite{MJB} Let $n,p_{1},p_{2},q_{1},q_{2},r_{1}$ and $r_{2}$ be positive integers, such that $n=p_{i}q_{i}+r_{i}$ and $r_{i}<p_{i}$, for $1\leq i\leq 2$. If $p_{1}<p_{2}$, then\\
$(p_{1}-r_{1}){q_{1} \choose 2}+r_{1}{q_{1}+1 \choose 2} \geq
(p_{2}-r_{2}){q_{2} \choose 2} +r_{2}{q_{2}+1 \choose 2}$.
\end{Lemma}
 We define distance
partition of $V(G)$ with respect to $v\in V(G)$, into distance
classes $\Psi_{i}(v)$ $(1\le i \le e(v))$ defined as:
$\Psi_{i}(v)=\{x\in V(G)| \: d(v,x)=i\}$ .
\begin{Proposition}\label{ObsrvFixDiferentClas}
Let $G$ be a distance transitive graph and $v,x,y\in V(G)$. Then
$x,y\in \Psi_i (v)$ for some $i$ $(1\le i \le e(v))$ if and only if
$v$ is non-adjacent to pair $(x,y)\in V_s(G)$ in $D(G)$.
\end{Proposition}
\begin{proof}
Let $x,y\in \Psi_i (v)$ for some $i$ $(1\le i \le e(v))$, then
$d(v,x)=d(v,y)=i$ and by definition of distance-transitive graph
there exist an automorphism $g\in Aut(G)$ such that $v\sim^g v$ and
$x\sim^g y$. Thus $x\sim^g y$ by an automorphism $g\in stab(v)$ and
consequently the pair
$(x,y)$ is not adjacent to $v$ in $D(G)$.\\
Conversely, suppose $v$ is non-adjacent to pair $(x,y)\in V_s(G)$,
then $x\sim^g y$ by an arbitrary $g\in stab(v)$. Since $g$ is an
isometry, therefore $d(v,x)=d(g(v),g(x))=d(v,y)=i$ (say). Thus $x,y$
are in same distance class $\Psi_i (v)$.
\end{proof}

\begin{Proposition}\label{PropDegvDistancTransKDetermined}
Let $G$ be a distance-transitive graph of order $n$. If $G$ is
$k$-fixed, then for each $v\in V(G)$, $
\emph{deg}_{D(G)}(v)={n\choose
2}-\sum_{i=1}^{e(v)}{|\Psi_{i}(v)|\choose 2}.$
\end{Proposition}
\begin{proof}
By Propositon \ref{ObsrvFixDiferentClas}, the only pairs $(x,y)\in
V_s(G)$ which are non-adjacent to $v\in V(G)$ are those in which
both $x,y$ belong to same distance class $\Psi_i(v)$ for each $i$
$(1\le i \le e(v))$. So the number of such pairs in $V_s (G)$ which
are not adjacent to $v$ is $\sum_{i=1}^{e(v)}{|\Psi_{i}(v)| \choose
2}$. Therefore, $\emph{deg}_{D(G)}(v)={n \choose
2}-\sum_{i=1}^{e(v)}{|\Psi_{i}(v)| \choose 2}$
\end{proof}
Thus, an expression for $|E(D(G))|$ can be obtained using
Proposition \ref{PropDegvDistancTransKDetermined},
\begin{equation}\label{eq10}
|E(D(G))|=\sum_{v\in V(G)}[{n \choose
2}-\sum_{i=1}^{e(v)}{|\Psi_{i}(v)| \choose 2}]\\ =n{n \choose
2}-\sum_{v\in V(G)} \sum_{i=1}^{e(v)}{|\Psi_{i}(v)| \choose
2}\end{equation}

From (\ref{eq9}) and (\ref{eq10}) we obtain
\begin{equation} \label{eq2}
n(k-1)\leq \sum_{v\in V(G)} \sum_{i=1}^{e(v)}{|\Psi_{i}(v)| \choose
2} \leq {n \choose 2} (k-1).
\end{equation}

\begin{Theorem}\label{thm3}
Let $G$ be a distance-transitive graph of order $n$ and diameter
$d$. If $G$ is $k$-fixed, then $k\geq \frac{n-1}{d}.$
\end{Theorem}
\begin{proof}
Note that, for each $v\in V(G)$,
$|\bigcup_{i=1}^{e(v)}\Psi_{i}(v)|=n-1$. For $v\in V(G)$, let
$n-1=q(v)e(v)+r(v)$, where $0\leq r(v)<e(v)$. Then, by Observation
\ref{obs3}, $\sum_{i=1}^{e(v)}{|\Psi _{i}(v)| \choose 2}$ is minimum
if and only if $|\;|\Psi_{i}(v)|-|\Psi_{j}(v)|\;| \leq 1$, where
$1\leq i,j \leq e(v)$. This condition will be satisfied if there are
$r(v)$ distance classes having $q(v)+1$ vertices and $e(v)-r(v)$
distance classes having $q(v)$ vertices.\\
Thus, the number of pair of vertices in $\Psi_{i}(v)$ having
$q(v)+1$ vertices is $r(v) {q(v)+1 \choose 2}$ and the number of
pair of vertices in $\Psi_{i}(v)$ having $q(v)$ vertices is
$(e(v)-r(v)){ q(v) \choose 2}$. Thus,
\begin{equation} \label{eq3}
(e(v)-r(v)){q(v) \choose 2}+r(v) {q(v)+1 \choose 2}\leq
\sum_{i=1}^{e(v)}{|\Psi _{i}(v)| \choose 2}.
\end{equation}
Let $w\in V(G)$ with $e(w)=d$, $r(w)=r$, and $q(w)=q$, then $n-1=qd+r$. Since for each $v\in V(G)$, $e(v)\leq e(w)$, by Lemma \ref{lma2},\\
$(d-r){q \choose 2}+r {q+1 \choose 2} \leq (e(v)-r(v)) {q(v) \choose 2}+ r(v) {q(v)+1 \choose 2}.$\\
Therefore,
\begin{displaymath}\label{eq4}
n[(d-r){q \choose 2} +r {q+1 \choose 2} ] \leq \sum_{v\in
V(G)}[(e(v)-r(v)){q(v) \choose 2}+r(v) {q(v)+1 \choose 2}.
\end{displaymath}
Thus, by relation (\ref{eq2}) and (\ref{eq3})
\begin{displaymath}
n[(d-r){q \choose 2} +r {q+1 \choose 2} ] \leq \sum_{v\in V(G)}
\sum_{i=1}^{e(v)} {|\Psi_{i}(v)| \choose 2} \leq {n \choose 2 }
(k-1).
\end{displaymath}
Hence, $q[(d-r)(q-1)+r(q+1)]\leq (n-1)(k-1)$, which implies,
$q[(r-d)+(d-r)q+r(q+1)]\leq (n-1)(k-1)$. Therefore,
$q(r-d)+q(n-1)\leq (n-1)(k-1)$. Since $q=\lfloor \frac{n-1}{d}
\rfloor$, we have
\[k-1\geq q+q \frac{r-d}{n-1}=q+\frac{qr}{n-1}-\frac{qd}{n-1}=q+\frac {qr}{n-1}-\frac {\lfloor \frac{n-1}{d}\rfloor d}{n-1}\geq q+ \frac{qr}{n-1}-1.\]
Thus, $k\geq \lfloor \frac{n-1}{d}\rfloor +\frac{qr}{n-1}$. Note
that, $\frac{qr}{n-1}\geq 0$. If $\frac{qr}{n-1}>0$, then $k\geq
\lceil \frac{n-1}{d} \rceil$, since $k$ is an integer. If
$\frac{qr}{n-1}=0$, then $r=0$ and consequently, $d$ divides $n-1$.
Thus, $\lfloor \frac{n-1}{d} \rfloor = \lceil \frac{n-1}{d} \rceil$.
Therefore, $k\geq \lceil \frac{n-1}{d} \rceil \geq \frac {n-1}{d}$.
\end{proof}

\end{document}